\documentclass[12pt]{article}
\usepackage{layout}
\usepackage[cm]{fullpage}
\usepackage{amsmath}
\usepackage{amsthm}
\usepackage{amsfonts}
\usepackage{amssymb}
\usepackage{bold-extra}
\usepackage{accents}
\usepackage{fontenc}
\usepackage{amsthm}
\usepackage[top = 1cm, bottom=1.5cm, left =1.5cm, right =1.5cm]{geometry}
\usepackage{mathrsfs}
\usepackage{color}
\usepackage{textcomp}
\usepackage{theoremref}
\usepackage{graphicx}
\usepackage[normalem]{ulem}
\usepackage[dvipsnames]{xcolor}
\usepackage[hang,flushmargin]{footmisc}

\newcommand{\D}{\mathcal{D}}

\newcommand{\sfrac}[2] {\mbox{$\frac{#1}{#2}$}}
\newcommand{\prob}[1]{\textrm{\textbf{#1}}}
\newcommand{\Q}{\prob{Q}}
\newcommand{\E}{\prob{E}}
\newcommand{\norm}[2]{|\hspace{-1.75pt}|\hspace{1pt}{#1}\hspace{1pt}|\hspace{-1.75pt}|\hspace{0.5pt}_{#2}  } 

\newcommand{\p}{\overline{p}}
\newcommand{\en}[1]{\enskip #1 \enskip}
\newcommand{\enleq}{\en{\leq}}
\newcommand{\engeq}{\en{\geq}}
\newcommand{\eneq}{\en{=}}
\newcommand{\fal}[1]{\enskip &#1 \enskip}
\newcommand{\al}[1]{&#1 \enskip}

\newtheoremstyle{mytheoremstyle1} 
    {15pt}                    
    {15pt}                    
    {\itshape}                   
    {}                           
    {\bfseries}                   
    {}                          
    {1em}                       
    {}  
\theoremstyle{mytheoremstyle1}
\newtheorem{Theorem}{Theorem}[section]
\newtheorem{thm}[Theorem]{Theorem}
\newtheorem{lem}[Theorem]{Lemma}

\newtheorem{cor}[Theorem]{Corollary}
\newtheorem{prop}[Theorem]{Proposition}

\newtheoremstyle{mytheoremstyle2} 
    {15pt}                    
    {15pt}                    
    {}                           
    {}                           
    {\itshape }                   
    {.}                          
    {1em}                       
    {}  
\theoremstyle{mytheoremstyle2}
\newtheorem*{pf}{Proof}

\newtheoremstyle{mytheoremstyle3} 
    {15pt}                    
    {15pt}                    
    {}                           
    {}                           
    {\bfseries}                   
    {.}                          
    {1em}                       
    {}  
\theoremstyle{mytheoremstyle3}
\newtheorem{definition}[Theorem]{Definition}

\numberwithin{equation}{section}

\begin{document}

\title{The largest fragment of a homogeneous fragmentation process}

\author{Andreas Kyprianou, Francis Lane and Peter M\"orters\footnote{
Department of Mathematical Sciences, University of Bath, Bath BA2 7AY, UK}}

\date{}

\maketitle

\begin{abstract}
\noindent We show that in homogeneous fragmentation processes the largest fragment at time $t$ has 
size $$e^{-t \Phi'(\p)}t^{-\frac32 (\log \Phi)'(\p)+o(1)},$$%
where $\Phi$ is the L\'evy exponent 
of the fragmentation process, and $\p$ is the unique solution of the equation \smash{$(\log \Phi)'(\bar{p})=\frac1{1+\bar{p}}$.} 
We argue that this result is in line with  predictions arising from  the classification of homogeneous fragmentation processes as logarithmically correlated random fields.
\end{abstract}



\section{Introduction}

There has been considerable interest in the past couple of years in a universality class of stochastic models called \emph{logarithmically correlated fields}. This class includes 
branching  Brownian motion~\cite{B78, R13, ABK13}, branching random walks~\cite{AR09, shi, A13}, 
the Gaussian free field on a planar lattice domain~\cite{D06, BZ12, BDZ16},  
the logarithmically correlated random energy model~\cite{FDR09}, 
Gaussian $1/f$-noise~\cite{FDR12}, 
nested conformal loops~\cite{A15}, and
Gaussian multiplicative chaos~\cite{RV14, M15} to name just a few. Plenty of interesting features arise 
from the conjectured membership of combinatorial and probabilistic objects such as eigenvectors of random matrix ensembles
in this class, conjectures of Fyodorov, Hiary and Keating~\cite{FHK12, FK14} on the maximum of the Riemann zeta function on an interval of the  critical line and of the characteristic polynomial of random unitary matrices being well-known examples, see~\cite{Arg16} for a survey. 
\medskip

Let us briefly describe some of the heuristic features of this class,  as sketched, for example, in~\cite{FG14}.
%
Characteristic  of these models is that, loosely speaking, at a large fixed level~$n$ they can be described as a centred field $(V(x) \colon x\in 2^{-n} {\mathbb Z}^d \cap  (0,1)^d)$ with correlations obeying a scaling of the type
\begin{equation}\label{var0}
\mathbb{E} \big[ V(x) V(y) \big] \sim - d \, \Psi''(0) \, \log |x-y|, \qquad \mbox{ if } \mbox{$2^{-n}$} \ll |x-y| \ll 1,
\end{equation}
where  $\Psi$ is a characteristic exponent given as
$$\mathbb{E}\big[ e^{pV(x)}\big] \sim 2^{dn(\Psi(p)-1)}.$$
The conjectured behaviour that the models in this universality class have in common relates to their extremal geometry. It has been argued (at varying levels of detail and rigour) that the highest peak at level~$n$
in a logarithmically correlated field satisfies 
\begin{equation}\label{P1}
\max_{x\in 2^{-n} {\mathbb Z}^d \cap  [0,1]^d}  V(x) \eneq \Psi'(\bar{q}) (d \log 2) n 
 - \frac32 (\log \Psi)'(\bar{q})  \log n + O(1)
\end{equation}
in probability, where $\bar{q}$ solves the equation \smash{$\Psi'(\bar{q})=\frac{\Psi(\bar{q})}{\bar{q}}$.}
In some cases finer results have been obtained, including the precise distribution of the asymptotic 
random constant of order one in the expansion of
$\max V(x)$ and fine results on the peaks seen from the largest peak, see for example~\cite{Br83, A13, ABK13}.
\medskip

An alternative approach to logarithmically correlated fields comes from the work of Fyodorov, Le Doussal and Rosso~\cite{FDR09}. They look at random fields satisfying a multifractal formalism and conjecture that, under natural conditions, the disorder-induced multifractality implies a
logarithmic scaling of the correlations. The highest peak then satisfies
\begin{equation}\label{P2}
\max_{x\in 2^{-n} {\mathbb Z}^d \cap  [0,1]^d}  V(x) \eneq \alpha_+ (d \log 2) \,  n  + \frac32  \big(f'(\alpha_+)\big)^{-1}  \log n +O(1),
\end{equation}
where $f(\alpha)= \dim\{ x \colon \lim \frac1n V(x) = \alpha (d \log 2) \big\}>0$
is the multifractal spectrum on the domain $(\alpha_-, \alpha_+)$ with boundary values given by $f(\alpha_-),f(\alpha_+)=0$. The multifractal formalism relates 
the spectrum to the characteristic exponent through the Legendre transform $\Psi(q)=\max\{ f(\alpha)+ q \alpha\}$.
\medskip

The purpose of the present paper is to align the class of homogeneous fragmentation processes with the universality class of logarithmically correlated fields by describing the processes' extremal behaviour and arguing that the rigorous result we obtain is consistent with the predictions obtained from the heuristics above. This makes homogeneous fragementation processes one of very few examples of a \emph{non-Gaussian}  field where the universality hypothesis can be verified. It also gives non-rigorous evidence that further properties of the class of of logarithmically correlated fields, such as convergence of the constant order term in \eqref{P2} to a random variable of a particular shape and existence of a freezing transition, also hold in this case, but we will not give technical proofs of this. 
\medskip%

Fragmentation processes represent the (typically) continuous splitting of an object into smaller parts.  We describe a fragmentation process by means of a random family $\{ {\mathcal I}^x(t) \colon x\in (0,1), t\geq 0\}$ of intervals such that ${\mathcal I}^x(t)\subseteq (0,1)$ is the interval containing $x$ at time~$t$. We assume that the following consistency relations are satisfied: 
\begin{enumerate}
\item $x\in {\mathcal I}^x(t)$ \hspace{2pt};
\item ${\mathcal I}^x(t)\subseteq  {\mathcal I}^x(s)$ if $s<t$ \hspace{2pt}; \enskip and
\item if $y\in{\mathcal I}^x(t)$, then ${\mathcal I}^x(t)={\mathcal I}^y(t)$.
\end{enumerate}
\pagebreak[3]

The random evolution of the fragmentation is given by a \emph{dislocation measure}~$\nu$ defined on the  partitions of the unit interval. Every interval $ {\mathcal I}^x(t)$ decomposes independently  at rate $\nu(du)$ into parts whose relative sizes are given by the partition~$u$. If the measure $\nu$ is finite then the process
$(\log I^x(t) \colon t>0)$, is a random walk for every $x\in(0,1)$  and the  entire system is a branching random walk. Our interest is therefore mostly on the case of infinite dislocation measure when both particle movement and branching become instantaneous and classical results on  branching random walk cannot be applied.  A rigorous definition of the process in the infinite dislocation measure case will be given in Section~2 of this paper.%
\medskip%

We let $\upsilon$ be uniformly distributed in the unit interval and define the L\'evy exponent, when it is finite, as
\[
\Phi(p) := - \log \E \big[ |{\mathcal I}^\upsilon(1)|^p \big]<\infty,
\]
or, equivalently, in terms of the dislocation measure as
\[
\Phi(p) = \int \bigg( 1 - \sum_{i = 1}^{\infty} |u_i|^{p+1} \bigg) \nu(du) ,
\]
where $(u_i: i \in \mathbb{N})$ is an enumeration of the partition sets  of $u$,
see~\cite{bas} for details. Our main result, Theorem~\ref{mainthm}, describes the size of the largest 
fragment at time $t\uparrow\infty$ as
\[
\max_{x\in[0,1]} |{\mathcal I}^x(t)|=e^{-t \Phi'(\bar{p})}t^{-\frac32 (\log \Phi)'(\bar{p})+o(1)},
\]
where \smash{$\bar{p}$} is the unique solution of the equation $(\log \Phi)'(\bar{p})=\frac1{1+\bar{p}}$.
\medskip


\section{Preliminaries and Main Result}

Before stating the main result of this paper, we briefly discuss the definition of conservative homogeneous interval fragmentation processes and some of their basic properties. An informal description of such a process is as follows.  The process starts from some initial configuration of fragments (i.e. subsets of $(0,1)$), which break up independently of one another as time passes. In general, these fragmentation events occur instantaneously in time. Looking at a single fragment at a given time, its subsequent evolution (after scaling to unit length) looks precisely the same as the fragmentation of any other (similarly scaled) particle. This means, in particular, that our fragmentations are time-homogeneous - the rate of `breaking up' is independent of particle size. Finally, we allow no loss of mass; the sum of the lengths of the fragments at any given time  equals the sum of the lengths of the fragments in the initial configuration.\medskip

Let us now briefly state the formal definition of a conservative homogeneous interval fragmentation process, referring to~\cite{bas} for proofs and further details. Let $\mathcal{U}$ denote the space of open subsets of $(0,1)$, which serves as our state-space. 
Each set $u \in \mathcal{U}$ has a unique decomposition into disjoint, non-empty, open intervals. The intervals comprising this decomposition are referred to as the \textit{fragments} or \textit{particles} of the set, and represent the `pieces' of the object that `falls apart at random'.
For $u,v \in \mathcal{U}$, we define the distance between $u$ and $v$ to be the Hausdorff distance between $(0,1)\setminus u$ and $(0,1) \setminus v$ (see \cite{bertselfsim}). We also endow $\mathcal{U}$ with the $\sigma$-algebra generated by the open sets corresponding to this distance, which we denote by $\mathcal{B}(\mathcal{U})$.  \medskip

Our basic data are a family  $(q_{t}\colon t > 0)$ of probability measures defined on $(\mathcal{U},\mathcal{B}(\mathcal{U}))$. We fix an interval $I:= (a,b) \subseteq (0,1)$ and write $\mathcal{I}$ for the set of open subsets of $I$ (with the distance inherited from $\mathcal{U}$ and the corresponding $\sigma$-algebra). We introduce the affine map $g_I : (0,1) \rightarrow I$, and retain the notation $g_I$ for its natural extension to a map from $\mathcal{U}$ to $\mathcal{I}$. We write $q^I_{t}$ for the image measure of $q_{t}$ under the map $g_I$, so that   $q^I_{t}$ is a probability measure on $\mathcal{I}$. Given an open set $u \in \mathcal{U}$ and a measurable enumeration $(u_i: i \in \mathbb{N})$ of the intervals in its decomposition, we write $q^u_{t} $ for the distribution of $\cup X_i$ where the $X_i$ are independent random variables with laws $q^{u_i}_{t}$ respectively. 

\begin{definition} \thlabel{fragproc} A Markov process $ U:=(U(t): t \geq 0)$ taking values in $\mathcal{U}$ is called a \textit{conservative homogeneous interval fragmentation} if it has the following properties:
\begin{enumerate}
\item $U$ is continuous in probability; 
\item $U$ is nested in the sense that $s>t \Rightarrow U(s) \subseteq U(t)$;
\item \textit{Fragmentation property}: there exists some family $(q_{t}: t > 0)$ of probability measures on $\mathcal{U}$ such that 
$$ \forall t \geq 0 \quad \forall s>t \quad \forall A \in \mathcal{B}(\mathcal{U}) \quad \text{\textnormal{\prob{P}}}\big(U(s) \in A \hspace{3pt} \big| \hspace{3pt}U(t)\big) = q^{U(t)}_{s-t}(A) ;$$ 
\item $|U(t)| = 1$ for all $t \geq 0$.  
\end{enumerate}
\end{definition}

The filtration generated by $U$ is denoted by $ \mathcal{F}:=(\mathcal{F}_t : t \geq 0)$, 
and the law of the  fragmentation started from $u \in \mathcal{U}$ by $\prob{P}_u$, with corresponding expectation operator $\prob{E}_u$. We define $\prob{P} := \prob{P}_{(0,1)}$ with expectation operator $\prob{E}$.\vspace{15pt}

Denoting by $u^*$ the largest interval component of $u \in \mathcal{U}$ we call a measure $\nu$ on $\mathcal{U}$ a \textit{dislocation measure}  if it satisfies $\nu((0,1)) = 0$ and
\vspace{-5pt}
\begin{equation} \label{integrability}
\int_{\mathcal{U}} \left( 1 - |u^*| \right) \nu(du) < \infty, 
\end{equation}
c.f. Definition 2.6 of \cite{bas}.  Given a homogeneous interval fragmentation we obtain a dislocation 
measure~$\nu$ by letting, for $u \in\mathcal{B}(\mathcal U)$ and $I=(0,1)$,
\[
	\nu(u)
\eneq 
	\lim_{t\downarrow 0} \frac1t \big( q_t^I(u)- q_0^I(u) \big).
\]

The measure $\nu$ is called the \textit{dislocation measure corresponding to} $U$, and it characterises the law of~$U$. 

\vspace{17pt}

Next we introduce the collection of \textit{tagged fragments}. Given a fragmentation process $U$ and $x \in (0,1)$, the \textit{x-tagged process} is simply the process of intervals in  $U$ containing $x$. We write $\mathcal{I}^x(t)$ for this fragment at time $t \geq 0$, and $ I^x(t)$ for its length. We also introduce the family of processes  $(\xi^x : x \in (0,1) ) $, where $\xi^x(t):= -\log I^x(t) $.  Letting $\upsilon$ denote a uniform random variable on $(0,1)$ which is independent of all the random variables introduced above, the processes $(\mathcal{I}_t),(I_t)$ and $(\xi_t)$ are defined by replacing $x$ with $\upsilon$ in the preceding definitions. These are the corresponding \textit{randomly tagged} processes.  Importantly, $\xi$ is a subordinator. We denote its Laplace exponent by $\Phi(p) := - \log \prob{E} ( e^{-p \xi(1)})$, which exists and is infinitely differentiable on the interval $(\underline{p},\infty)$ for some $\underline{p} \in [-1,0]$. \bigskip

Using the concavity of $\Phi$, it is easy to show that the equation 
\begin{equation} \label{overbar}   \frac{\Phi(p)}{1+p} = \Phi'(p) 
\end{equation} 
has a unique solution  $\overline{p} \in (\underline{p}, \infty)$, and that this solution is positive. The value $\overline{p}$ has great importance in the present context. For instance, with $c_{\overline{p}} := \Phi'(\overline{p})$, we have
\[
\lim_{t \rightarrow \infty} \frac{\inf_{x \in (0,1)} \xi^x_t}{t} = c_{\overline{p}}  \quad \textrm{a.s.},
\]
giving the first term in the asymptotic expansion of the size of the largest particle (see, for example, \cite{BER01}). \\

We are now ready to state the main result of this paper, which identifies the second term of this asymptotic expansion in terms of  $\overline{p}$ :

\begin{thm} \thlabel{mainthm}
Starting from any initial configuration in $\mathcal{U}$, 
$$\frac{\inf_{x \in (0,1)} \xi^x(t) -  c_{\overline{p}}t}{\log t} \enskip \longrightarrow \enskip \frac{3}{2}(\overline{p} + 1)^{-1} \enskip =: l \quad \textrm{in probability as} \enskip t \uparrow \infty.$$
\end{thm}

The proof is based on martingale methods and is close in spirit to that of~\cite{shi}. Roughly speaking, we will define random variables that count the number of particles that are too large or small. Using two tools - a Many-to-One Lemma  
and a change of measure (to be introduced shortly) - we will estimate the moments of these random variables using the fluctuation theory of L\'evy processes. \bigskip

To be precise, let us introduce the processes $\zeta^x_t := \xi^x_t - c_{\overline{p}}t$ for each $x \in (0,1)$, $t \geq 0$, and the corresponding randomly tagged process $\zeta_t := \xi_t - c_{\overline{p}}t$ for $t \geq 0$. For each  $p > \underline{p} $ we also define the process $(\mathcal{E}_p (t) : t \geq 0)$
by  
\[
\mathcal{E}_p (t) := \exp \big( \Phi(p)t - p\xi_t \big).
\] 
This process is a unit mean $(\mathcal{F},\prob{P})$-martingale, allowing us to define the family of probability measures $\big( \prob{Q}^{p} : p > \underline{p})$  by  
\[
\frac{d {\prob{Q}}^{p}}{d \prob{P}} \bigg\vert_{\mathcal{F}_t} 
\eneq 
\mathcal{E}_p (t)  \qquad \mbox{ for } t \geq 0.
\]

In fact we will only use $\Q := \Q^{\overline{p}}$.  This is because, as a consequence of the equation defining $\p$, \eqref{overbar}, the spectrally positive L\'evy process $(\zeta, \Q)$ has zero mean. It is also well-known that $\zeta$ has finite moments of all orders under $\Q$. These special properties allow us to use  results on L\'evy processes with zero mean and finite variance, which are collected in the appendix.  \bigskip

For a set $A \subset (0,1)$, we use the notation $\sum_{[x]_t : A}$ to represent sums taken over the (countable) collection of \textit{distinct} fragments at time $t$ that are subsets of $A$.
We also write $\sum_{[x]_t }$ for $\sum_{[x]_t : (0,1)}$, the sum taken over \textit{all} distinct fragments at time $t$.
For a Borel set $B \subset \mathbb{R}$, $|B|$ stands for the Lebesgue measure of $B$.  Using this notation, we make the simple observation that for any $u \in \mathcal{U}$, $t \geq 0$, and measurable non-negative function $F$ on paths of tagged fragments, we can write
\begin{align*}
\text{\textnormal{\prob{E}}}_u \sum_{[x]_t} F( \xi^x_s : s \leq t)
\fal{=}
\sum_{i=1}^{\infty} \prob{E}_u  \sum_{[x]_t : u_i}   F( \xi^x_s : s \leq t) \\
\al{=} 
\sum_{i=1}^{\infty} \prob{E}  \sum_{[x]_t : (0,1)}   F( \xi^x_s  - \log |u_i|: s \leq t)  \\
\al{=}
\sum_{i=1}^{\infty} \text{\textnormal{\prob{E}}} \big(  I^{-1}_t F( \xi_s - \log |u_i| : s \leq t) \big) ,
\end{align*}

where the sums in $i$ should be regarded as finite in case $u$ consists of finitely many blocks. To illustrate the notation, $\sum_{[x]_t : u_i}$ sums over distinct particles at time $t$ which result from the fragmentation of the interval $u_i$. In the second equality we have used the fragmentation property. To get from second to third, we introduce the factor $I^x_t \cdot (I^x_t)^{-1}$ inside the second sum, which can then be interpreted as a size-biased pick.  Proceeding to make the change of measure $\prob{E} \rightarrow \Q$, we obtain the following Many-to-One lemma:

\begin{lem} \emph{(MT1)} \thlabel{mt1}
For any measurable, non-negative function $F$ on paths of tagged fragments and any $u = (u_1, u_2, ...) \in \mathcal{U}$ we have
\[ \text{\textnormal{\prob{E}}}_{u} \sum_{[x]_t} F( \zeta^x_s : s \leq t) = \sum_{i=1}^{\infty} \text{\textnormal{\prob{Q}}} \big( e^{\zeta_t(\overline{p}+1)}  F( \zeta_s - \log |u_i| : s \leq t) \big) ,\]
In particular, 
\[
\text{\textnormal{\prob{E}}} \sum_{[x]_t} F( \zeta^x_s : s \leq t) =  \text{\textnormal{\prob{Q}}} \big( e^{\zeta_t(\overline{p}+1)}  F( \zeta_s  : s \leq t) \big) .
\]
\end{lem}

To prove \thref{mainthm}, it suffices to prove the following statement two statements for arbitrary $u \in \mathcal{U}$:

\begin{equation}
\prob{P}_u \left( \inf_{x \in (0,1)} \zeta^x(t) \leq \alpha \log t \hspace{2pt}  \right) \rightarrow 0 \quad \textrm{as} \quad t \uparrow \infty \enskip  \textrm{ for all } \enskip \alpha < l \hspace{3pt}; \quad \label{liminf}
\end{equation} 
\begin{equation}
\limsup_{t \rightarrow \infty} \frac{ \inf_{x \in (0,1)} \zeta^x(t)}{\log t} \leq l \hspace{2pt} \quad \prob{P}_u-\textrm{almost surely.}  \label{limsup}
\end{equation}

The structure of the remainder of the paper is as follows. In Section~$3$ we prove \eqref{liminf}, and in Section~$4$ we prove~\eqref{limsup}, the more challenging result. The arguments are analogous to those in \cite{shi}, but there are significant differences on the technical level, occuring particularly in the proof of \eqref{limsup}. The analogous part of the proof in \cite{shi} makes certain moment assumptions that are not satisfied in our framework. We do not need these moment assumptions, as we are able to exploit the special features of our fragmentation processes - namely, that particles decrease in size, and no mass is lost. 
In Section~5 we align our result with heuristics on logarithmically correlated fields. Our proof relies on fine results on  L\'evy processes, which are provided in the appendix, Section~$6$.

\vspace{20pt}


\section{Proof of \eqref{liminf}}

Fix an arbitrary $\alpha \in (0,l)$, $k \in \mathbb{N}$ and  $u = (u_1,u_1,...) \in \mathcal{U}$. Define, for $t \geq 0$, the random variable 
\begin{equation} \label{sumrv} Z^k_t :=  \sum_{[x]_t} \textbf{1} \Big( \zeta^x_t \leq \alpha \log t, \enskip \underline{\zeta}^x_t \geq -k \Big),
\end{equation} 
where $\underline{\zeta}^x_t := \inf_{0 \leq s \leq t} \zeta^x_{\hspace{1pt} s}$. This random variable counts the number of `bad' particles (with a truncation we will remove later). \medskip

We estimate the mean of $Z^k_t$ under $\prob{E}_u$ as follows, recalling that $\zeta$ is the randomly tagged process corresponding to the family of processes $\left( \zeta^x : x \in (0,1) \right)$:
\begin{align} \label{firstest}
\prob{E}_u Z^k_t
\enskip &= \enskip 
\sum_{i=1}^{\infty} \Q \Big( e^{\zeta_t(\overline{p}+1)}  \textbf{1} \big( \zeta_t - \log |u_i| \leq \alpha \log t, \enskip \underline{\zeta}_t - \log |u_i| \geq -k \big) \Big) \nonumber \\ 
&\leq \enskip
t^{\alpha(\overline{p}+1)} \sum_i |u_i|^{\overline{p}+1} \Q \Big(   \zeta_t - \log |u_i| \leq \alpha \log t, \enskip \underline{\zeta}_t - \log |u_i| \geq -k  \Big).
\end{align}

In the first line we use MT1 (Lemma~\ref{mt1}) , and in the second we bound the exponential factor using the indicator. Recalling that $(\zeta, \Q)$ is a spectrally positive L\'evy process with zero mean and finite variance, we can estimate a typical probability on the right-hand side of the previous inequality using \thref{shicor1}:
\begin{align} \label{secondest}
\Q \Big(   \zeta_t - \log |u_i| \leq \alpha \log t, \enskip \underline{\zeta}_t - \log |u_i| \geq -k  \Big) \enskip &\leq \enskip 
\gamma \, t^{-3/2} (k-\log |u_i| +1)(k+\alpha \log t)^2 \nonumber \\
&\leq \enskip \gamma_k \, t^{-3/2}  (\log t)^2 (1- \log|u_i|),
\end{align}

for some constants $\gamma, \gamma_k >0$ (where the latter depends on $k$). Putting this back into \eqref{firstest}, we find that 
\begin{equation}
\prob{E}_u Z^k_t
\enskip \leq \enskip
\gamma_k  t^{\alpha(\overline{p}+1)} t^{-3/2} (\log t)^2 \sum_i |u_i|^{\overline{p}+1} (1-\log |u_i|).
\end{equation}

Since $\overline{p} > 0$, the function $x \mapsto x^{\overline{p}} (1-\log x) $ has an upper bound $K >0$ on $(0,1)$, so the sum on the right-hand side is bounded by $K \sum |u_i| = K$. We deduce that
\begin{equation}
\prob{E}_u Z^k_t
\enskip \leq \enskip K \gamma_k \: t^{\alpha(\overline{p}+1)} t^{-3/2} (\log t)^2.
\end{equation}

Since $\alpha(\overline{p}+1) < l(\overline{p}+1) = 3/2$, this quantity goes to zero as $t \rightarrow \infty$.

\vspace{10pt}

To complete this part of the proof we must remove the truncation $\underline{\zeta}^x_t \geq -k$ in \eqref{sumrv}. To this end, we introduce the 
\emph{intrinsic additive martingale} corresponding to $\overline{p}$, 
\[
M_t := e^{\Phi(\overline{p}) t} \sum_{[x]_t} I^x(t)^{1+\overline{p}} 
= \sum_{[x]_t} \exp \big( -(1+\overline{p}) \hspace{2pt} \zeta^x_t
\big) .
\]
By the martingale convergence theorem, $M_t$ converges to a finite limit $\prob{P}_{u}$-almost surely as $t \rightarrow \infty$. Noting that $\overline{p} > 0$, we get 
$\inf_{t \geq 0} \inf_{x \in (0,1)} \zeta^x_t >  -\infty$  $\prob{P}_{u}$-a.s.
Letting $B_k :=  \big\lbrace  \inf_{t \geq 0} \inf_{x \in (0,1)}   \zeta^x_t \geq  -k \big\rbrace$ 
for each $k \in \mathbb{N}$, it follows that
\begin{equation} \label{truncation2}
\lim_{k \rightarrow \infty} \prob{P}_u (B_k) = 1 .
\end{equation}

Next fix an arbitrary $\epsilon > 0$, and (using \eqref{truncation2}) select $k=k(\epsilon) \in \mathbb{N}$ so large that $\prob{P}_{u}(B_k) \geq 1 - \epsilon$. Observing that  $Z^k_t  \geq \textbf{1}_{B_k}
  \sum_{[x]_t} \textbf{1} (\zeta^x_t \leq \alpha \log t)$ for all $t \geq 0$, we may then write, 
\begin{align} \label{truncation3}
\prob{P}_{u}(Z^k_t = 0) \enskip &\leq \enskip \prob{P}_{u} \bigg[ B_k \cap \Big\{  \sum_{[x]_t} \textbf{1} \big( \zeta^x_t \leq \alpha \log t \big) = 0 \Big\}\bigg] \enskip + \enskip \prob{P}_{u} (B_k^c) \nonumber \\
&\leq \enskip \prob{P}_{u}\bigg[ \sum_{[x]_t} \textbf{1} \big(\zeta^x_t \leq \alpha \log t \big)  = 0 \bigg] \enskip + \enskip \epsilon,
\end{align}

for all $t \geq 0$. We have already shown that $\prob{E}_{u} (Z^k_t) \rightarrow 0$ as $t \rightarrow \infty$, and so, since $Z^k_t$ takes values in $\{0,1,2,...\}$, we deduce that $\prob{P}_{u} (Z^k_t = 0) \rightarrow 1$ as $t \uparrow \infty$.  Combining this observation with \eqref{truncation3} we conclude that  

$$ 1 \enskip = \enskip \liminf_{t \rightarrow \infty} \prob{P}_{u} (Z^k_t = 0)  \enskip \leq \enskip \liminf_{t \rightarrow \infty}\prob{P}_{u}\bigg[ \sum_{[x]_t} \textbf{1} \big(\zeta^x_t \leq \alpha \log t \big) \hspace{3pt} = 0 \bigg]  \enskip + \enskip \epsilon \enskip . $$

Since $\epsilon > 0$ was arbitrary, we get
$1=\lim_{t \rightarrow \infty}\prob{P}_{u} \Big[ \sum_{[x]_t} \textbf{1} \big(\zeta^x_t \leq \alpha \log t \big) \enskip  = 0 \Big]$.
Finally, observe that  
$$ \bigg\lbrace \sum_{[x]_t} \textbf{1} (\zeta^x_t \leq \alpha \log t)\enskip  = 0 \bigg\rbrace \enskip \subseteq \enskip 
\left\lbrace \inf_{x \in (0,1)} \zeta^x_t > \alpha \log t \right\rbrace,$$ 
so that 
$\prob{P}_{u} \big( \inf_{x \in (0,1)} \zeta^x_t > \alpha \log t \big) 
\rightarrow 1$ as $t \uparrow \infty,$
which implies that \eqref{liminf} holds. \qed

\vspace{20pt}


\section{Proof of \eqref{limsup}}

In this part of the proof, we can work under $\prob{P}$ without loss of generality. To see why, note that we are now trying to show the existence of `big' particles (in the sense made precise by \eqref{limsup}). This means that, starting the fragmentation from general $u \in \mathcal{U}$, we can immediately look only at the largest particle at time $t$ descending from $u^*$, whose size we call $B^{u^*}_t$. Let $B_t$ denote the size of the largest fragment at time~$t$ in a fragmentation issued from $(0,1)$. The fragmentation property implies that $(B^{u^*}_t , \prob{P}_u)$ is equal in law to $(|u^*| B_t, \prob{P})$. The numerator in 
\eqref{limsup} corresponding to these two processes will therefore only differ by the additive constant $-\log |u^*| $, which goes to zero in the limit upon division by $\log t$. 

\vspace{10pt}

Let $C>0$ be the larger of the two constants provided by \thref{nonzeroliminf} and \thref{shiprop1}. Introduce the following intervals:
\begin{equation}
	J_s(t)
\en{:=} 
	\begin{cases} 
		[-1, \infty) &\mbox{if} \enskip 0 \leq s \leq t \\ 
		[l \log t , \infty) &\mbox{if} \enskip t < s < 2t \\
		[l \log t, l \log t + 2C] &\mbox{if} \enskip s = 2t. 
	\end{cases} 
\end{equation}

For $x \in (0,1)$ and $u,v \in [0,2t]$, define the events $A^x_{2t,[u,v]} := \{ \zeta^x_s \in J_s(t) \enskip \forall s \in [u,v] \}$, and write $A^x_{2t}:= A^x_{2t,[0,2t]}$. In what follows, $A_{2t}$ (with no superscript) means $A^{\upsilon}_{2t}$, where $\upsilon$ is the uniformly distributed random tag in $(0,1)$ in the definition of $\zeta$. Finally, define the random variable
$Z_t := \sum_{[x]_{2t}} \textbf{1}_{A^x_{2t}}$. \medskip

The first step is to bound $\E Z_t$ from below. Using MT1 (Lemma~\ref{mt1}), 
we obtain
\begin{align*} \label{lower}
	\prob{E} Z_t 
\eneq 
	\Q 
	\big( 
	e^{\zeta_t(\overline{p}+1)} \,
	\textbf{1}_{A_{2t}}
	\big) 
\engeq
	\gamma \,  t^{3/2} \, \Q(A_{2t})  
\engeq 
	\gamma' 
\en{>} 
	0,
\end{align*}
for some $\gamma, \gamma' > 0$ and all large $t$. In the first inequality we have used the indicator to bound the exponential factor from below; the second uses \thref{nonzeroliminf}.\medskip 

Next, we bound the second moment of $Z_t$ from above. To this end we introduce the notation $\mathcal{D}_s$ to denote the random set of all fragmentation times in $[0,s]$, which, in general, is almost surely dense in $[0,s]$. For $r \in \D_{s}$ write $B_{[z]_r}$ for the event that the interval $\mathcal{I}^z_{r-}$  shatters at time $r$. Note that for $r \in \D_s$ precisely one of the indicators 
\smash{$\textbf{1}_{B_{[z]_{r-}}} $} over all dinstinct fragments $[z]_{r-} \subset (0,1)$
takes the value $1$ (simultaneous fragmentations of distinct blocks is a null event). We then make the decomposition
\begin{equation} \label{decomp}
Z_t^2 = Z_t + \Lambda_t,
\end{equation}

where
\begin{align} \label{defLambda} 
	\Lambda_t 
\fal{:=} 
	\sum_{r \in \mathcal{D}_{2t}} \hspace{3pt}   
	\sum_{[z]_{r-} : (0,1)} \textbf{1}_{A^z_{[0,r-]}} \hspace{4pt} 		
	\textbf{1}_{B_{[z]_r}}
	\sum_{\substack{[x]_r , [y]_r : [z]_{r-} \\ [x]_r \neq [y]_r}}  
	\hspace{4pt}
	\sum_{\substack{[u]_{2t} : [x]_{r} \\ [v]_{2t} : [y]_{r}}}
	\textbf{1}_{A^u_{[r,2t]}} \textbf{1}_{A^v_{[r,2t]}}\\ 
\al{=} 
	\sum_{r \in \mathcal{D}_{2t}} \hspace{3pt}   
	\sum_{[z]_{r-} : (0,1)}
	\Lambda^z_{r},
\end{align}

where the second line defines $\Lambda^z_{r}$. As we are temporarily regarding $t$ as fixed, we have written $A^w_{[u,v]}$ for  $A^w_{2t,[u,v]}$. This decomposition is similar to the one used in \cite{shi}, but we have the added complication that the sum in $r$ is over a random (dense) set. To explain this decomposition, first note that $Z_t^2 = \sum_{[u]_{2t}} \textbf{1}_{A^u_{2t}} \cdot \sum_{[v]_{2t}} \textbf{1}_{A^v_{2t}}$. The $Z_t$ in \eqref{decomp} comes from the terms in this product where $\mathcal{I}^u_{2t} = \mathcal{I}^v_{2t}$. When  $\mathcal{I}^u_{2t} \neq \mathcal{I}^v_{2t}$, we find their most recent common ancestor $\mathcal{I}^z_{r-} $ just before it fragments (at time $r$) into the distinct ancestors $ \mathcal{I}^x_{r}$ and $ \mathcal{I}^y_{r}$ of $\mathcal{I}^u_{2t}$ and $\mathcal{I}^v_{2t}$ respectively. \medskip

Our aim is to bound $\prob{E} \Lambda_t $ from above. The first part of the calculation uses the fragmentation property to make the summand indexed by $r$ in \eqref{defLambda} measurable with respect to $\mathcal{F}_r$.  To this end, we first show that, for all $s>0$, the set $\mathcal{D}_s$ almost surely has an enumeration $(r_1,r_2,...)$ with the property that each $r_i$ is an $\mathcal{F}$-stopping time. 
Fix $s>0$, and a strictly increasing (deterministic) sequence $(a_i) \subset [0,1)$ with $a_1 = 0$ and $\lim a_i = 1$. 
If $[z]_{r-}$ (for some $z \in (0,1)$) is the particle that shatters at time $r \in \mathcal{D}_s$, then the fragments at time $r$ resulting from this fragmentation event are given by an affine image of some $u_r \in \mathcal{U}$. We write $u_r^*$ for the largest interval component of $u_r$.
We then introduce the sets $\mathcal{D}_{s,n} := \{r \in \mathcal{D}_s :  |u^*_r| \in [a_n,a_{n+1}) \}$. 
Of course, $\mathcal{D}_s = \bigcup_{n \in \mathbb{N}} \mathcal{D}_{s,n}$, and, as we will now show, $D_{s,n} := \# \mathcal{D}_{s,n} < \infty$ almost surely, for all $n \in \mathbb{N}$. To this end, we rewrite $D_{s,n}$ as follows:
\begin{equation*}
	D_{s,n}
\eneq
	\sum_{0 \leq r \leq s} \textbf{1}_{B_{[z]_r}}  
	\textbf{1}_{\big( \, |u^*_r| \in [a_n,a_{n+1}) \big) \,} .
\end{equation*}
Using the compensation formula (see page 99 of \cite{KYP14}), we deduce that $\E D_{s,n} = s \, \nu \big( u^* \in [a_n,a_{n+1}) \big)$. It remains to note that, for all $n \in \mathbb{N}$,
\begin{equation*}
	\nu \big( u^* \in [a_n,a_{n+1}) \big) 
\enleq
	(1-a_{n+1})^{-1} \int_{\mathcal{U}} \left( 1 - |u^*| \right) \nu(du)
\en{<}
	\infty.
\end{equation*}
The desired enumeration is then obtained by listing the elements of each (almost surely finite) set $\mathcal{D}_{s,n}$ in order of increasing size, and concatenating the resulting sequences.\medskip

Using the enumeration $(r_1,r_2,...)$ constructed above (with $s=2t$) and the non-negativity of the terms in \eqref{defLambda}, we can now take the first step towards estimating $\E \Lambda_t$, writing
\begin{equation} \label{enum}
	\E \Lambda_t 
\eneq 
	\sum_{i = 1}^{\infty}  
	\E \sum_{[z]_{r_i- }:(0,1)} \Lambda^z_{r_i}
\eneq
	\sum_{i = 1}^{\infty}  
	\E \hspace{2pt} \E_{\mathcal{F}_{r_i}}  		
	\sum_{[z]_{r_i- }:(0,1)} \Lambda^z_{r_i} .
\end{equation}

In the second equality we have conditioned the term in the sum labelled by $r_i$ on the sigma-algebra $\mathcal{F}_{r_i}$. Next we calculate these conditional expectations. Fixing $r = r_i$ for some $i \in \mathbb{N}$, we have
\begin{equation}
	\prob{E}_{\mathcal{F}_{r}} 
	\sum_{[z]_{r-} : (0,1)}
	\Lambda^z_{r}
\eneq
	\sum_{[z]_{r-} : (0,1)} \textbf{1}_{A^z_{[0,r-]}} \hspace{4pt} \textbf{1}_{B_{[z]_r}}
	\sum_{\substack{[x]_r , [y]_r : [z]_{r-} \\ [x]_r \neq [y]_r}}
	\hspace{4pt}
	\E_{\mathcal{F}_r}
	\sum_{\substack{[u]_{2t} : [x]_{r} \\ [v]_{2t} : [y]_{r}}}
	\textbf{1}_{A^u_{[r,2t]}} \textbf{1}_{A^v_{[r,2t]}}.
\end{equation}

where we have used the fact that $r$ is $\mathcal{F}_r$-measurable. We then write 
\begin{equation} \label{indep}
\E_{\mathcal{F}_r}
\sum_{\substack{[u]_{2t} : [x]_{r} \\ [v]_{2t} : [y]_{r}}}
\textbf{1}_{A^u_{[r,2t]}} \textbf{1}_{A^v_{[r,2t]}}
\eneq
\left( \E_{\mathcal{F}_r}
\sum_{[u]_{2t} : [x]_{r}}
\textbf{1}_{A^u_{[r,2t]}}  \right)
\left( \E_{\mathcal{F}_r}
\sum_{[v]_{2t} : [x]_{r}}
\textbf{1}_{A^v_{[r,2t]}} \right)
\end{equation}
for $x,y \in (0,1)$ such that $\mathcal{I}^x_{2t} \neq \mathcal{I}^y_{2t}$, using the independent evolution of distinct particles. Now we calculate a typical factor on the right-hand side of \eqref{indep} (explanations follow the calculation): 
\begin{align}
	\E_{\mathcal{F}_r}
	\sum_{[u]_{2t} : [x]_{r}}
	\textbf{1}_{A^u_{[r,2t]}}  
\fal{=}
	\E_{\mathcal{F}_r}
	\sum_{[u]_{2t} : [x]_{r}}
	\textbf{1}_{ \big( 
	\zeta^u_s \in J_s(t) \enskip \forall s \in [r,2t] 
	\big)} 
	\nonumber\\
\al{=} 
	\E_{\mathcal{F}_r}
	\sum_{[u]_{2t} : [x]_{r}}
	\frac{I^u_{2t}}{I^x_r} \frac{I^x_r}{I^u_{2t}} 
	\textbf{1}_{ \big( 
	\zeta^u_s \in J_s(t) \enskip \forall s \in [r,2t] \big)} 
	\nonumber\\
\al{=}
	\left( 
	\E \,
	I_{2t-r}^{-1} 
	\textbf{1}_{ \big( 
	\alpha + \zeta_s \in J_{s+r}(t) \enskip \forall s \in [0,2t-r] 
	\big)} 
	\right) 
	\bigg\vert_{\alpha = \zeta^x_r} 
	\nonumber \\
\al{=}
	\left( \Q \,
	e^{\zeta_{2t-r}(\p + 1)} \textbf{1}_{ \big( 
	\alpha + \zeta_s \in J_{s+r}(t) \enskip \forall s \in [0,2t-r] 
	\big)} 
	\right) 
	\bigg\vert_{\alpha = \zeta^x_r} \nonumber  
\en{=} 
	F(\zeta^x_r),
\end{align} 

where, for $\alpha \in \mathbb{R}$,
\begin{equation*}
	F(\alpha) 
\en{:=} 
	\Q \,
	e^{\zeta_{2t-r}(\p + 1)} \textbf{1}_{ \big( \alpha + \zeta_s \in J_{s+r}(t) \enskip \forall s \in 		    [0,2t-r] \big)} .
\end{equation*}

In the first line we just write down the definition of the events $A^u_{[r,2t]}$; the second artificially introduces a size-based pick; the third makes use of the size-biased pick together with the fragmentation property; and the final line makes the change of measure $\E \rightarrow \Q$. So far, we've shown that 
\begin{equation} 
	\prob{E}_{\mathcal{F}_{r}} \sum_{[z]_{r-} : (0,1)} 
	\Lambda^z_{r}
\eneq
	\sum_{[z]_{r-} : (0,1)} \textbf{1}_{A^z_{[0,r-]}} \hspace{4pt} \textbf{1}_{B_{[z]_r}}
	\sum_{\substack{[x]_r , [y]_r : [z]_{r-} \\ [x]_r \neq [y]_r}}
	F(\zeta^x_r) F(\zeta^y_r).
\end{equation}

Putting this expression back into \eqref{enum} and exchanging the summation and expectation, we arrive at
\begin{equation} \label{reduction}
	\E \Lambda_t
\eneq 
	\E \sum_{r \in \mathcal{D}_{2t}} \hspace{3pt}   
	\sum_{[z]_{r-} : (0,1)} \textbf{1}_{A^z_{[0,r-]}} \hspace{4pt} \textbf{1}_{B_{[z]_r}}
	\sum_{\substack{[x]_r , [y]_r : [z]_{r-} \\ [x]_r \neq [y]_r}}
	F(\zeta^x_r) F(\zeta^y_r).
\end{equation}

We have now succeeded in making the $r$-indexed summand $\mathcal{F}_r$-measurable, which will allow us to use the compensation formula (see page 99 of \cite{KYP14}). 
To this end, define the function $G: \mathbb{R} \times \mathcal{U} \rightarrow [0,\infty]$ by
\[
G(\alpha, u) \en{:=} \sum_{a_u \neq b_u} F(\alpha - \log |a_u|) F(\alpha -\log |b_u|)
\]
where the sum is over the distinct interval components $a_u, b_u \subset u$.\footnote{
The function $G$ can be constructed in a measurable way by ordering the interval components of $u \in \mathcal{U}$ in order of decreasing length, $(u_1,u_2,...)$, and then writing the sum as $\sum_{i=1}^{\infty} \sum_{j\neq i}  F(\alpha - \log |u_i|) F(\alpha -\log |u_j|) $.} 
Using the compensation formula, we can move from \eqref{reduction} to
\begin{align} \label{precomp}
	\E \Lambda_t 
\fal{=} 
	\int_0^{2t} dr \, \cdot \, 
	\E \left(
	\sum_{[z]_{r-} : (0,1)} \textbf{1}_{A^z_{[0,r-]}}  
	\int_{\mathcal{U}} G(\zeta^z_{r-},u) \nu (du) 
	\right) \nonumber \\
\al{=}
	\int_0^{2t} dr \, \cdot \, 
	\Q \left(
	e^{\zeta_{r-}(\p+1)}
	\textbf{1}_{A_{[0,r-]}}  
	\int_{\mathcal{U}} G(\zeta_{r-},u) \nu (du) 
	\right) \nonumber \\
\al{=} 
	\int_0^{2t} dr \, \cdot \, \lambda(r),	
\end{align}

where the final equality defines $\lambda(r)$ 
as the integrand of the previous line. Here, $\nu$ is the dislocation measure introduced in Section $2$, which satisfies the integrability condition \eqref{integrability}. \bigskip

\textbf{Notation:} In the remainder of this section, positive constants (independent of $t$) will be denoted by $\gamma > 0$, the value of which will change from one inequality to another.  \bigskip

We state the next part of the proof as a lemma:

\begin{lem} \hspace{10pt} $\displaystyle{\text{\textnormal{\prob{E}}} \Lambda_t = \int_0^{2t} \lambda(r) dr = O\big((\log t)^3 \, \big)}$ \quad as \enskip $\displaystyle{t \uparrow \infty}$.
\end{lem}

\begin{proof}

First we estimate $F(\alpha - \log |a_u|)$ for interval components $a_u$ of $u \in \mathcal{U}$ and $\alpha \in \mathbb{R}$: using the indicator to bound the exponent we have
\begin{align} \label{estF}
	F(\alpha - \log |a_u|)
\fal{=}
	\Q \,
	e^{\zeta_{2t-r}(\p + 1)} \textbf{1}_{ \big( \alpha - \log |a_u| + \zeta_s \in J_{s+r}(t) \enskip \forall        	s \in [0,2t-r] \big)} \nonumber \\
\al{\leq}
	\gamma \, t^{3/2} \, \, |a_u|^{\p + 1}e^{-\alpha (\p +1)}
	f(\alpha - \log |a_u|),
\end{align}
for some $\gamma > 0$, with 
\[
f(\theta) \en{:=} \Q\bigg(\theta + \zeta_s \in J_{s+r}(t) \enskip \forall s \in [0,2t-r] \bigg), \qquad \textrm{for} \enskip \theta \in \mathbb{R}.
\]

We estimate $f$ in two different ways, depending on the value of $r$. For $r \in [t,2t]$, \thref{shiprop1} provides the estimate
\begin{equation*}
f(\theta) \enleq \Q \bigg( \zeta_{2t-r} \in [l \log t - \theta, l \log t - \theta + 2C] \bigg) 
\enleq \gamma \: n_{2t-r},
\end{equation*}
with $n_{\theta} := \theta^{-1/2} \wedge 1$ for $\theta \geq 0$. Referring back to \eqref{precomp}, this leads to the bound
\begin{equation} \label{firstLambda}
	\int_t^{2t} \lambda(r) \, dr
\enleq
	\gamma \,  I_1 \, t^3  	
	\int_0^{2t} dr \, \cdot 
	n_{2t-r}^2
	\Q \left(
	e^{-\zeta_{r}(\p+1)}
	\textbf{1}_{A_{[0,r-]}}  
	\right)  \, ,
\end{equation}

where
\[
I_1 \en{:=} \int_{\mathcal{U}} \nu(du) \, \cdot \sum_{a_u \neq b_u} |a_u|^{\p+1} |b_u|^{\p+1} \, .
\]

Let us check that $I_1$ is finite. Indeed, 
\begin{align*}
	\sum_{a_u \neq b_u} |a_u|^{\p+1} |b_u|^{\p+1} 
\fal{\leq}
	\sum_{a_u \neq b_u} |a_u| |b_u| 
\eneq
	\sum_{a_u} |a_u| (1-|a_u|)\\
\al{\leq} 
	(1-|u^*|) + \sum_{a_u \neq u^*} |a_u| \\
\al{=} 
	2 (1-|u^*|).
\end{align*}
In the first inequality we use the facts that $|a_u|,|b_u| < 1$ and $\p > 0$; in the first equality we fix an interval component $a_u$ of $u \in \mathcal{U}$ and sum over the interval components $b_u \neq a_u$ of $u$; and in the second inequality we use the fact that $|a_u| \in (0,1)$.  
The finiteness of $I_1$ then follows from \eqref{integrability}. It remains to estimate the expectation in \eqref{firstLambda}:
\vspace{-3pt}
\begin{align*}
	\prob{Q} \big(  e^{-\zeta_{r-} (\overline{p}+1)} \textbf{1}_{A_{[0,r-]}} \big) 
\fal{\leq} 
	\gamma \, t^{-3/2} \,  
	\prob{Q}  
	\big(  
	\textbf{1}_{A_{[0,r-]}} \textbf{1}_{( \zeta_{r-} \leq 2 \,l \log t )}  
	\big) 
	+ 
	\gamma \, \prob{Q} 
	\big(  
	e^{-\zeta_{r-} (\overline{p}+1)}   
	\textbf{1}_{A_{[0,r-]}} 
	\textbf{1}_{( \zeta_{r-} > 2 \,l \log t )}
	\big) \\[3pt]
\al{\leq} 
	\gamma \, t^{-3/2} \Q\big( \underline{\zeta}_{r-} \geq - 1 , \: \zeta_{r-} \leq  2 l \log t \big) + 		\gamma \, t^{-3}\\[3pt]
\al{\leq}  
	\gamma \, t^{-3/2} (r^{-3/2} \wedge 1) (\log t)^2 + \gamma \, t^{-3}.
\end{align*}
In the first line we split the event $\{ \zeta_{r-} \geq l \log t\} \subset A_{[0,r-]}$ into the events $\{ \zeta_{r-} > 2 l \log t\}$ and $\{ l \log t \leq \zeta_{r-} \leq 2 l \log t\}$. In the second line, we discard some information from the indicator on the interval $[t,r]$ and estimate the exponential factor in the second term using the indicator  $\textbf{1}_{( \zeta_{r-} > 2 \,l \log t )}$. In the final line, we use \thref{shicor1} to estimate the remaining expectation. 
Returning to \eqref{firstLambda}, we conclude that
\begin{equation*}
	\int_t^{2t} \lambda(r) \, dr
\enleq 
	\int_t^{2t} \, dr \, \cdot 	
	\left[	
	\gamma \, (\log t)^2 \, t^{3/2} \, 
	(r^{-3/2} \wedge 1) \, n_{2t-r}^2 
	\en{+} 
	\gamma \, n_{2t-r}^2
	\right].
\end{equation*}

Elementary analysis allows us to conclude that $\displaystyle{\int_t^{2t} \lambda(r) dr = O\big((\log t)^3\big),}$ as required. \bigskip

Now we look at $\lambda(r)$ for $r \in [0,t]$. This time we make the estimate 
\vspace{-5pt}
\begin{align*}
	f(\theta) 
\fal{\leq} 
	\Q \left( 
		\underline{\zeta}_{2t-r} \geq -1 - \theta, 
		\enskip 
		\zeta_{2t-r} \in [l \log t - \theta, l \log t - \theta + 2C] \, 	
	\right) \\
\al{\leq} 
	\gamma \, (1+\theta) \, (\log t) \, (2t-r)^{-3/2} \\
\al{\leq} 
	\gamma \, (1+\theta) \, (\log t) \,  t^{-3/2}.
\end{align*}

In the first inequality we throw away some information from the indicator on the interval $[t,2t-r)$; in the second we use \thref{shicor1}; and the final inequality uses the fact that $r \in [0,t]$. Making the substitution $\theta = \alpha - \log |a_u|$, we arrive at
\begin{align*}
	f(\alpha - \log |a_u|)  
\fal{\leq}
	\gamma \, (1+\alpha - \log |a_u| ) \, (\log t) \, t^{-3/2}\\
\al{\leq}
	2 \gamma \,(2+\alpha) \, ( 1 - \log |a_u| ) \, (\log t) \, t^{-3/2}
\end{align*}
for $\alpha \geq -1$ (recall we intend to make the substitution $\alpha = \zeta_{r-} \geq -1$). This leads to the bound
\[
	\lambda(r)
\enleq
	\gamma \,  I_2 \, (\log t)^2 \,	
	\Q \left(
	e^{-\zeta_{r}(\p+1)}
	(2+\zeta_{r-})^2
	\textbf{1}_{A_{[0,r-]}}  
	\right)  ,
\]

where
\[
	I_2 
\en{:=} 
	\int_{\mathcal{U}} \nu(du) \, \cdot \sum_{a_u,b_u} |a_u|^{\p+1} |b_u|^{\p+1}
	( 1 - \log |a_u| )
	( 1 - \log |b_u| ) .
\]

This time we note that the function $x \mapsto x^{\p} (1 - \log x)$ is bounded on $[0,1]$, since $\p > 0$. This allows us to write $I_2 \leq K \int_{\mathcal{U}} \nu(du) \cdot  \sum |a_u| |b_u|$ (for some $K>0$), which is finite by the same arguments we used for $I_1$. To complete the proof we define $\tau_0^- := \inf \{s \geq 0 : \zeta_s < 0 \} $, and let $\Q_1$ denote the law of $1+\zeta_t$ under $\Q$. We note then that
\begin{align*}
\int_0^t dr \cdot \Q \, \left( (2+\zeta_{r-})^2 e^{-\zeta_{r-}(\p + 1)} \textbf{1}_{A_{[0,r-]}} \right)
\fal{\leq} 
e^{\p + 1} \, \Q_1 \int_0^{\tau_0^-}  (1+\zeta_{r-})^2 e^{-\zeta_{r-}(\p + 1)} \: dr \, .
\end{align*}
Defining the function $h: [0,\infty) \rightarrow [0,\infty)$ by $h(\theta) := (1+\theta)^2 \, e^{-(\p+1)\theta}$, and bearing in mind that $\zeta$ is spectrally positive, we apply Theorem 20 (page 196) of \cite{BER96} to make the following calculation:
\begin{align*}
	\Q_1 \int_0^{\tau_0^-}  (1+\zeta_{r-})^2 e^{-\zeta_{r-}(\p + 1)} \: dr 
\fal{=}
	\Q_1 \int_0^{\tau_0^-}  h(\zeta_{r-})\: dr  \\
\al{=} 
	\gamma \int_0^{\infty} dy \, \cdot \,
	\int_0^1 dz \, \cdot \, 
	h(1+y-z)
\end{align*}
for some $\gamma > 0$. It remains to note that the right-hand side of the previous display is bounded by $K \int_0^\infty e^{-w} dw < \infty$, for some finite constant $K>0$ (since $1+\p > 1$). 
\end{proof}

\vspace{10pt}

Let us collect together the facts we have established in this section so far: for some $\gamma_1, \gamma_2 >0$, we have
\begin{align}
&\prob{E} (Z_t) \geq \gamma_1 \enskip ; &&\hspace{-120pt}\textrm{and}  \label{fact1}\\
&Z_t^2 \eneq Z_t + \Lambda_t \enskip , &&\hspace{-120pt}\textrm{with}  \label{fact2}\\ 
&\E \Lambda_t \enskip \leq \enskip \gamma_2 (\log t)^3 \label{fact3} ,
\end{align}

for all large $t$. Following page $7$ of \cite{shi}, we make the following simple calculation, valid for all large $t$: 
\begin{equation} \label{fact4}
\prob{E} (Z_t^2) \enskip \leq \enskip \gamma_2 (\log t)^3 + \prob{E}(Z_t)  \enskip \leq \enskip  \left[ \sfrac{\gamma_2}{\gamma_1}(\log t)^3 + 1 \right]  \prob{E}(Z_t) \enskip \leq \enskip  \left[ \sfrac{\gamma_2}{\gamma_1}(\log t)^3 + 1 \right]  \sfrac{1}{\gamma_1} \prob{E} (Z_t)^2 ,
\end{equation}

where the first inequality uses \eqref{fact2} and \eqref{fact3}, and the next two inequalities use \eqref{fact1}. First making use of the Paley-Zygmund inequality, 
and then of \eqref{fact4}, we find that 

\begin{equation*}
\prob{P} (Z_t > 0) \enskip \geq \enskip \frac{\prob{E}(Z_t)^2}{\prob{E}(Z_t^2)} 
\enskip \geq \enskip \frac{\gamma}{(\log t)^3} \hspace{5pt}.
\end{equation*}

We then note that 

\begin{equation*}
\left\lbrace \min_{x \in (0,1)} \zeta_{2t} \enskip > \enskip l\log t + 2C\right\rbrace \enskip \subseteq \enskip \left\lbrace Z_t = 0 \right\rbrace 
\end{equation*}

so that, for all sufficiently large $t$, we have
\begin{align}
\prob{P} \left\lbrace \min_{x \in (0,1)} \zeta^x_t \enskip > \enskip l\log t + 2C\right\rbrace  \enskip \fal{\leq} \enskip  \prob{P} \left\lbrace \min_{x \in (0,1)} \zeta^x_t \enskip > \enskip l\log \sfrac{t}{2} + 2C\right\rbrace  \nonumber \\
\al{\leq} \enskip  1 - \frac{\gamma}{(\log t)^3} .
\end{align}

Now we need to know the rate at which the number of exceptionally large particles grows. 
To be precise define, in the notation of \cite{krell}, sets
$G_{c,\alpha,\beta}(t) := \{ \mathcal{I}^x(t) : x \in (0,1) , \alpha e^{-ct} < I^x(t) < \beta e^{-ct} \}$
for $0<\alpha <1< \beta$ and $c \in \mathbb{R}$. A result from~\cite{bertoin} shows that for $c \in (c_{\overline{p}}, \Phi'(\underline{p}+)) $ there exists $\tilde{\rho}(c)>0$, depending only on $c$, and not on $\alpha$ or $\beta$, such that
$$ \lim_{t \rightarrow \infty} \frac{1}{t} \log \#G_{c,\alpha,\beta}(t) =  \tilde{\rho}(c) \quad \textrm{a.s.}   $$

We fix a small $\delta>0$ and $c:= c_{\overline{p}}+\delta$,
and define sets $\mathcal{N}(t):= \{ \mathcal{I}^x(t) \colon \xi^x_t - ct \leq  1 \}  $. We deduce that for $\rho=\tilde\rho(c)>0$ we have
\begin{equation} \label{rate1}
\lim_{t \rightarrow \infty} \frac{1}{t} \log \#\mathcal{N}(t) \geq  \rho \quad \prob{P}-a.s.
\end{equation}

\medskip

Next, fix an arbitrary $\epsilon>0$ and define $T_n := T(n,\epsilon) := \inf \{ t \geq 0 : \#\mathcal{N}(t) \geq n^{\epsilon} \}$.  We choose the $\lfloor n^{\epsilon} \rfloor$ largest elements of $\mathcal{N}(T_n)$ 
and label them $\{ \mathcal{I}^{n,j} : 1 \leq j \leq \lfloor n^{\epsilon} \rfloor \}$ in order of increasing size.  We then write $\xi^{n,j,x}_t$ to denote the $-\log$ of the size of the particle containing $x \in \mathcal{I}^{n,j}$ at each time $t \geq T_n$. Note, for instance, that $ \xi^{n,j,x}_{T_n} = - \log I^{n,j}$  for all $x \in \mathcal{I}^{n,j} $.  For all $n \in \mathbb{N}$ we have 
\[
\prob{P} \left(  \max_{s \in [\frac{n}{2},n] \cap \mathbb{N}}  \enskip \min_{1 \leq j \leq \lfloor n^{\epsilon} \rfloor}  \enskip \min_{x \in \mathcal{I}^{n,j}} \enskip \xi^{n,j,x}_{ T_n + s} -  c_{\overline{p}}(T_n + s) >  \max_{1 \leq j \leq \lfloor n^{\epsilon} \rfloor} \xi^{n,j} -c_{\overline{p}} T_n  + l \log n + 2C   \right)
\]

\vspace{-10pt}

\begin{align} \label{nexpression} 
&\leq 
	\enskip \sum_{s \in [\frac{n}{2},n] \cap \mathbb{N}} \prob{P} \left(  \inf_{x \in (0,1)} \xi^x_s - 	
	c_{\overline{p}}s  \enskip > \enskip l \log t + 2C   \right)^{\lfloor n^{\epsilon} \rfloor }  
\leq \enskip 
	\sum_{s \in [\frac{n}{2},n] \cap \mathbb{N}} \left(   1 - \frac{\gamma}{(\log s)^3}  
	\right)^{\lfloor n^{\epsilon} \rfloor } \nonumber \\[3pt]
&\leq \enskip 
	\frac{n}{2} \left(   1 - \frac{\gamma}{(\log n)^3}  \right)^{ n^{\epsilon} -1}.
\end{align}

The final expression is summable in $n$ (see \thref{sum}). By the Borel-Cantelli lemma, we deduce that, $\prob{P}$-almost surely,
\begin{align}
\max_{s \in [\frac{n}{2},n] \cap \mathbb{N}}  \enskip \min_{1 \leq j \leq \lfloor n^{\epsilon} \rfloor}  \enskip \min_{x \in \mathcal{I}^{n,j}} \enskip \xi^{n,j,x}_{ T_n + s} -  c_{\overline{p}}(T_n + s)  \enskip &\leq \enskip  
\max_{1 \leq j \leq \lfloor n^{\epsilon} \rfloor} \xi^{n,j} -c_{\overline{p}} T_n  + l \log n + 2C   \nonumber \\
&\leq \enskip 1 + (c_{\overline{p}}+\delta) T_n - c_{\overline{p}} T_n  + l \log n + 2C \nonumber  \\[5pt]
&= \enskip \delta T_n + l \log n + 2C + 1  . \label{maxminmax}
\end{align}

The final ingredient we need to finish the proof is to show that 

\begin{equation} \lim_{n \rightarrow \infty} \frac{T(n,\epsilon)}{\log n} \enskip \leq  \enskip  \frac{\epsilon}{\rho} \quad \prob{P}-\textrm{a.s.}  \label{rate}
\end{equation}

To do this, fix $\epsilon' \in (0, \rho)$. Then, by \eqref{rate1},  there is some almost-surely finite random variable $T \geq 0$ such that, almost surely, $t \geq T$ implies $\#\mathcal{N}(t) \geq e^{(\rho - \epsilon')t}$. Consequently, for all $t \geq T$ we know that 
$$ T(n,\epsilon) \enskip \leq \enskip \inf \Big\lbrace t \geq 0 : e^{(\rho - \epsilon')t} \geq n^{\epsilon} \Big\rbrace \enskip = \enskip \frac{\epsilon \log n } {\rho - \epsilon'}  .$$
This yields \eqref{rate}. Combining \eqref{maxminmax} and \eqref{rate} we find that $\prob{P}$-almost surely, for all large~$n$, we have 
\begin{equation*}
\max_{s \in [\frac{n}{2},n] \cap \mathbb{N}}  \enskip \min_{1 \leq j \leq n}  \enskip \min_{x \in \mathcal{I}^{n,j}} \enskip \xi^{n,j,x}_{ T_n + s} -  c_{\overline{p}}(T_n + s) \enskip \leq
 \enskip \left( \sfrac{2\epsilon \delta}{\rho}+ l \right)   \log n + 2C + 1.
\end{equation*}

By \eqref{rate}, we can write almost surely that $T_n = \epsilon \cdot O( \log n)$. We immediately deduce that for all large $n$ we have
\[
	\inf_{x \in (0,1)} \xi^x_{n+ T_n  } - c_{\overline{p}}(n+ \epsilon \, O(\log n))  
\enleq 
	\left( \sfrac{2\epsilon \delta}{\rho}+l \right) \log n + 2C  + 1   
	\hspace{20pt} \prob{P}- \text{a.s.} 
\]
Noting that $T_n \geq 0$ and that $t \mapsto \xi^x_t$ is monotonically increasing, we deduce that, for all large $n$,
\[
	\inf_{x \in (0,1)} \xi^x_n - c_{\overline{p}} n 
\enleq 
	\left( \sfrac{2\epsilon \delta}{\rho}+l \right) \log n  + \epsilon \, O(\log n)
	\hspace{20pt} \prob{P}- \text{a.s.} 
\]
Since $\epsilon>0$ can be made arbitrarily small, we conclude that 

$$ \limsup_{\mathbb{N} \ni n \rightarrow \infty} \frac{ \inf_{x \in (0,1)} \xi^x_n - c_{\overline{p}}n }{\log n} \enskip \leq \enskip   l  \hspace{20pt}  \prob{P}-\text{a.s.} $$

By monotonicity of $t\mapsto \xi^x_t$ we see that the limit can be taken through all real 
values,  completing the proof of~\eqref{limsup}. \quad \qed

 \vspace{20pt}
 

\section{Physical heuristics}

In this section we argue \emph{informally} how our result can be put in line with the predictions described 
for logarithmically  correlated random fields in the introduction. To cast the model  into this framework we let
$$V(x)=\log |{\mathcal I}^x(t)| - \mathbb{E} \log |{\mathcal I}^x(t)|, \qquad \mbox{ for }
x\in 2^{-n} {\mathbb Z} \cap  [0,1] \mbox{ and } t=\frac{n \log 2}{\Phi'(0)}, $$
where the choice of time scale comes from matching the spatial scale $2^{-n}$ to $e^{-t\Phi'(0)}$, which is the typical length of a 
tagged fragment at time~$t$ and hence the scale on which $V$ needs to be sampled.
We have 
\begin{align*}
e^{-t\Phi(p)} & =\mathbb{E}\big[ \big| \mathcal{I}^{\upsilon}(t)\big|^{p} \big]
\approx \sum_{x \in 2^{-n}\mathbb{Z} \cap(0,1)} \mathbb{E}  \big| \mathcal{I}^{x}(t)\big|^{p+1}
\approx \exp\big(   (p+1) \mathbb{E} \log | \mathcal{I}^\upsilon(t)| + (n \log 2) \Psi(p+1)\big).
\end{align*}
Observing that $\mathbb{E} \log |{\mathcal I}^\upsilon(t)|\sim- t\Phi'(0)$ we get
$$\Psi(p+1) =  1 + p - \frac{\Phi(p)}{\Phi'(0)}.$$
We introduce, for $x,y \in (0,1)$, the stopping time $T = T(x,y) := \inf \{t \geq 0: \mathcal{I}^x_t \neq \mathcal{I}^y_t \} $, the time when $x$ and $y$ are first split apart. Fixing an arbitrary
$t>0$ and abbreviating $\tau = t \wedge T$ we can decompose 
$$ |{\mathcal I}^x(t)| = 
|{\mathcal I}^x(\tau -)| \times \Delta^x_{\tau} \times |\tilde{\mathcal I}^{\tilde x}(t-\tau)|,$$
where $\Delta^x_s := |{\mathcal I}^x(s)|/|{\mathcal I}^x(s-)|$, the process 
$(\tilde{\mathcal I}^x(s) \colon s \geq 0)$ is a fragmentation process, which is independent of what happened up to time~$t$, and $\tilde x\in(0,1)$ is the relative position of $x$ in ${\mathcal I}^x(t)$.
Taking $\log$ on both sides of the decomposition  and centering gives
$V(x,t) \sim V(x,\tau -) + \tilde{V}(\tilde{x},t-\tau),$ as $t\uparrow\infty$,
where we define $V(z,t)= \log |{\mathcal I}^z(t)| - \mathbb{E} \log |{\mathcal I}^z(t)|$. 
Taking expectations and using the independence we get
$\mathbb{E} [ V(x,t) V(y,t) \big] \sim \mathbb{E}  \big[ V(x,\tau -) V(y, \tau -)]  .$
Using Wald's identity (see Theorem 3 of \cite{hall70}) we calculate the expectation on the right and obtain
$$
\mathbb{E} \big[ V(x,t) V(y,t) \big] \sim   \mathbb{E}\big[ t \wedge T\big] \,  \Phi'(0)\Psi''(0).
$$
Recalling that $t=\frac{n \log 2}{\Phi'(0)}$ we look at a regime where
$$t\Phi'(0)=n \log 2 \gg -\log |x-y|.$$ Observe that the right-hand side is at least
$-\log |{\mathcal I}^x(T-)| \sim T \Phi'(0)$ and hence
$\mathbb{E}[ t \wedge T]\sim \mathbb{E}[T]$. We obtain
\begin{equation}\label{var1}
\mathbb{E} \big[ V(x) V(y) \big] \sim  \mathbb{E}\big[ T(x,y) \big] \,  \Phi'(0)\Psi''(0) 
\qquad \mbox{ if } \mbox{$2^{-n}$} \ll |x-y| \ll 1.
\end{equation}
This is a result of the type~\eqref{var0}, if the distance 
of points $x,y$ on the interval is measured not with the euclidean metric, but
with respect to the natural random metric coming from our problem, defined by
$d(x,y)=|\mathcal{I}^x(T(x,y)-)|$ and therefore $-\log d(x,y)\sim {\Phi'(0)T(x,y)}.$
The result can also be partially claimed for the euclidean set-up,
as $\log |x-y| \leq \log |{\mathcal I}^x(T(x,y)-)| \sim \log d(x,y) \, \Phi'(0)$ but we will see below
that working in this framework will lead to a loss of accuracy.
\medskip

The physicist's prediction~\eqref{P1} hence gives 
$$\max_x \log |\mathcal{I}^x(t)|  + t \Phi'(0) \approx  \Psi'(\bar{q}) n \log 2 - \frac32 (\log \Psi)'(\bar{q}) \log n.$$
Recalling that $\Psi(p+1)=1+p-\frac{\Phi(p)}{\Phi'(0)}$ we get $\bar{q}=\bar{p}+1$ and hence
$$\max_x \log |\mathcal{I}^x(t)|  \approx -\frac{\Phi'(\bar{p})}{\Phi'(0)} n \log 2 - \frac32 \frac1{\bar{p}+1} \log n \sim -\Phi'(\bar{p}) t - \frac32 \frac1{\bar{p}+1} \log t,$$
which is in line with our rigorous result.
\medskip

To relate our story to the multifractal approach of Fyodorov, Le Doussal and 
Rosso~\cite{FDR09} we first recall the multifractal spectrum for homogeneous fragmentation processes obtained by Berestycki~\cite{b03}
and refined by Krell~\cite{krell}. We define 
$q_\beta$ by $\Phi'(q_\beta)=\beta$. Then,
for every $\beta$ making the right-hand side below positive, almost surely,
$$\dim_{|\cdot|}\big\{ x\in (0,1) \colon \lim_{t\uparrow\infty} -\sfrac1t \log |\mathcal{I}^x(t)| = \beta\big\}
= 1+ q_\beta -\frac{\Phi(q_\beta)}{\beta}.$$
Perhaps surprisingly, this formula does \emph{not} put our result in line with the prediction of Fyodorov, Le Doussal and Rosso. 
The prediction can however be reconciled with our results, if one moves to
the appropriate metric, which in our case is again the random metric $d$. While for fixed intervals the ratio of
lengths with respect to $d$ and the Euclidean metric are typically bounded from zero and infinity, the optimal coverings 
implicit in the Hausdorff dimension above use random intervals for which these diameters
are radically different. Indeed, given $\beta$ the covering intervals $I$ for the corresponding set have metric
diameters given by their length to the power $\Phi'(0)/\beta$ (see for example~\cite{M09}). As a result the
multifractal spectrum in the intrinsic random metric becomes
$$\dim_{d}\big\{ x\in (0,1) \colon \lim_{t\uparrow\infty} -\sfrac1t \log |\mathcal{I}^x(t)| = \beta\big\}
= \frac{\beta}{\Phi'(0)}\Big(1+ q_\beta\Big)  -\frac{\Phi(q_\beta)}{\Phi'(0)}.$$
This can be translated as
$$\dim_{d}\big\{ x\in (0,1) \colon V(x) \approx \alpha (\log 2) n \big\}
=\Psi(p_\alpha)- \alpha p_\alpha=: f(\alpha),$$
where $p_\alpha$ is given by $\Psi'(p_\alpha)=\alpha$. Hence $f'(\alpha)=-p_\alpha$. The right end of the spectrum, $\alpha_+$, is characterised by the equation $\Psi(p_{\alpha_+})=p_{\alpha_+}\Psi'(p_{\alpha_+})$, hence $p_{\alpha_+}=\bar q$ and $\alpha_+=\Psi'(\bar q)$ aligning the prediction of~\eqref{P2} with our result.

\vspace{20pt}


\section{Appendix on L\'evy Processes}

In this section we extend the lemmas found in the appendix of~\cite{shi} from random walks to L\'evy processes with finite variance and zero mean. 
The proofs proceed by contradiction: we assume that the various statements do not hold for appropriate L\'evy processes, and then generate a random walk contradicting the results in \cite{shi} by discretization. 
We begin by stating two elementary lemmas which will be of use in carrying out such arguments. The first is a topological lemma whose proof can be found in \cite{king}. The second is a simple observation, recorded for convenience. 
Throughout this section we write $X$ for the process $(X_t)_{t \geq 0}$.

\begin{lem} \label{topologicallemma} \quad Let $U\subseteq[0,\infty)$ be open and unbounded. Then there exists $h > 0$ such that $nh \in U$ for infinitely many~$n \in \mathbb{N}$. 
\end{lem}

\begin{lem} \thlabel{rcont}
\quad Let $X$
be a real-valued stochastic process issued from zero with almost surely right-continuous paths. Then 
\[
\forall \epsilon > 0 \quad \forall \delta > 0 \quad \exists \hspace{1pt} a>0 
\quad\mbox{ such that } \quad
\text{\textnormal{\prob{P}}}\big(  {\norm{X}{[0,a]}} > \delta \big)\hspace{2pt} < \hspace{2pt} \epsilon \hspace{1pt},
\]
where ${\norm{X}{[0,a]}} := \sup_{0 \leq t \leq a} |X_{\hspace{1pt} t} |$ .
\end{lem}

Now we state the first of our results on L\'evy processes.

\begin{prop}  \thlabel{shiprop1} Let $X$ be a  L\'evy process with zero mean and finite variance. Then 
$$ \exists C_{0}>0 \quad \exists c>0 \quad\mbox{ such that }\quad \forall h \geq C_{0} \quad \forall t > 0  \qquad   \sup_{r \in \mathbb{R}} \hspace{2pt} 
\text{\textnormal{\prob{P}}} \big(r \leq X_t \leq r+h \big)\enskip \leq \enskip c \hspace{2pt} \frac{h}{t^{1/2}} \hspace{5pt}. $$
\end{prop}

\begin{pf} Assume the above statement is not true, i.e. for some such L\'evy process $X$ 

\begin{equation}\forall n \in \mathbb{N} \qquad \exists \hspace{1pt} h_n \geq n \qquad \exists t_n > 0 \quad \exists r_n \in \mathbb{R} \quad\mbox{ such that }\quad
\prob{P}\big(r_n \leq X_{t_n} \leq r_n+h_n \big)\enskip > \enskip n \hspace{2pt} \frac{h_n}{t_n^{1/2}} \hspace{5pt}. \label{contrhyp}
\end{equation}

Now select an $a>0$ corresponding to the choices $\epsilon=\frac{1}{2}$ and $\delta = 1$ in  \thref{rcont}. Evidently, for all $n \in \mathbb{N}$,
\begin{align*}
\prob{P}\big( r_n -1 \leq X_{t} \leq r_n+h_n +1 \enskip \forall t \in [t_n,t_n+a] \big) \enskip
&\geq \enskip \prob{P}\big( r_n  \leq X_{t_n} \leq r_n+h_n , \enskip \norm{X_t -X_{t_n}}{{t \in [t_n,t_n+a]}} \hspace{2pt} < \hspace{2pt} 1 \big)\\
&\geq \enskip \frac{1}{2} \hspace{2pt} \prob{P} \big(  r_n \leq X_{t_n} \leq r_n+h_n \big) \enskip
\geq \enskip \frac{n}{2}\hspace{2pt} \frac{h_n}{t_n^{1/2}} \hspace{5pt},
\end{align*}
where in the second inequality we have used the Markov property of the L\'evy process at time $t_n$. Let $U:= \bigcup_{n=1}^\infty (t_n,t_n+a)$, which is an open set. Note that, to prevent the probability in \eqref{contrhyp} exceeding one we must have $t_n \geq n^4$, proving that $U$ is unbounded. 
Lemma~6.1 therefore supplies an $h>0$ and two strictly increasing sequences $(m_j)$ and $(n_j)$ of natural numbers with the property that, for all $j \in \mathbb{N}$ we have $m_{j}h \in [t_{n_j}, t_{n_j}+a]$. Note that ${t_{n_j}}/{m_j} \to h$ as $j \to \infty$. In particular, there exists $K>0$ such that 
\smash{${{K}/{m_j^{1/2}}}< {1}/{t_{n_j}^{1/2}}$} for all $j \in \mathbb{N}$.
Now define a random walk on $\mathbb{R}$ by $S_n := X_{nh}$, and note that this random walk has zero mean and finite variance. We estimate
\begin{align*}
\prob{P}\big(r_{n_j}-1 \leq S_{m_j} \enskip \leq \enskip r_{n_j}+h_{n_j}+1\big)  \enskip \fal{\geq} \enskip
\prob{P}\big( r_{n_j}-1 \leq X_{t} \leq r_{n_j}+h_{n_j}+1 \quad \forall t \in [t_{n_j},t_{n_j}+a] \big) \\
\al{\geq} \frac{K}{2} n_j \frac{h_{n_j}}{m_j^{1/2}}. 
\end{align*}

Taking suprema and assuming  without loss of generality that $h_{n_j}\geq 2$ for all $j \in \mathbb{N}$, we find that, for all  $j \in \mathbb{N} $,
$$ \sup_{r \in \mathbb{R}} \prob{P}(r \leq S_{m_j} \leq r+h_{n_j}+2) \quad \geq \quad
{\frac{K}{4} n_j \frac{h_{n_j}+2}{m_j^{1/2}}}, $$
contradicting (A.1) in \cite{shi}. \qed
\end{pf} 

\vspace{5pt}

\begin{prop} \thlabel{shiprop2}
Let $X$ be a  L\'evy process with zero mean and finite variance. Then, with $\underline{X}_t := \inf_{0 \leq s \leq t}X_{\hspace{1pt} s}$, we have 
$$ \limsup_{t \to \infty} \enskip
{t^{1/2}} \enskip
\sup_{u \geq 0} \enskip \frac{1}{u+1} \enskip
 \text{\textnormal{\prob{P}}} \big( \hspace{1pt}\underline{X}_t \geq -u \big) < \infty \enskip .$$
 \end{prop}

\begin{pf}  The statement in the proposition is equivalent to the following statement:
$$ {\exists C>0 \quad \exists T > 0 \quad \mbox{ such that} \quad t \geq T \Rightarrow \enskip \sup_{u \geq 0} \enskip \frac{1}{u+1} \prob{P}\big( \hspace{1pt}\underline{X}_t \geq -u \big) \enskip \leq \enskip \frac{C}{t^{1/2}}} \enskip .  $$

For a contradiction, let us assume the converse of this statement holds. Then
$$ {\forall n \in \mathbb{N} \quad \exists t_n \geq n \quad \exists u_n \geq 0 \quad \mbox{ such that} \quad 
 \frac{1}{u_n+1} \prob{P}\big( \hspace{1pt}\underline{X}_{\hspace{1pt}{t_n}} \geq -u_n \big) \enskip \geq \enskip \frac{n}{t_n^{1/2}}} \enskip . $$

 As in \thref{shiprop1}, select $a>0$ with the following property:  
$$ \frac{1}{u_n+1} \prob{P} \big( \hspace{1pt}\underline{X}_{\hspace{1pt}{t}} \geq -u_n-1 \quad \forall 
t \in [t_n,t_n+a] \big) \enskip \geq \enskip \frac{n}{2 t_n^{1/2}} \enskip . $$

Now choose sequences $(m_j)$ and $(n_j)$, and $K>0$ precisely as in the proof of \thref{shiprop1}. Select furthermore an $M>0$ with the property that $\frac{1}{u} \leq \frac{M}{u+1} \quad \forall u \geq 1$ . Defining the random walk $(S_n)_{n \in \mathbb{N}}$ as in \thref{shiprop1}, we estimate
\begin{align*}
\frac{K}{2} n_j {\frac{1}{m_j^{1/2}}}  \quad
&\leq \quad \frac{1}{u_{n_j}+1} \enskip \prob{P} \big( \hspace{1pt}\underline{X}_{\hspace{1pt}{t}} \geq -u_{n_j}-1 \quad \forall 
t \in [t_{n_j},t_{n_j}+a] \big) \enskip
\leq  \enskip
 \frac{1}{u_{n_j}+1} \enskip \prob{P} \big( \hspace{1pt}\underline{S}_{\hspace{1pt}{m_{j}}}\geq -u_{n_j}-1 \big) \enskip \\
 &\leq \enskip \sup_{u \geq 0} \enskip \frac{1}{u+1} \enskip \prob{P}\big( \hspace{1pt}\underline{S}_{\hspace{1pt}{m_{j}}}\geq -u-1 \big)
= \enskip \sup_{u \geq 1} \enskip \frac{1}{u} \enskip  \prob{P}\big( \hspace{1pt}\underline{S}_{\hspace{1pt}{m_{j}}}\geq -u \big)\\
&\leq \enskip M \enskip \sup_{u \geq 1} \enskip \frac{1}{u+1} \enskip  \prob{P}\big( \hspace{1pt}\underline{S}_{\hspace{1pt}{m_{j}}}\geq -u\big) \enskip
\leq \enskip M \enskip \sup_{u \geq 0} \enskip \frac{1}{u+1} \enskip  \prob{P}\big( \hspace{1pt}\underline{S}_{\hspace{1pt}{m_{j}}}\geq -u \big)\enskip . 
\end{align*}
This contradicts (A.3) in \cite{shi}. \qed
\end{pf}

\vspace{5pt}

With \thref{shiprop1} and \thref{shiprop2} in hand, the proof of the following corollary follows verbatim from the proof of Lemma A.1 of \cite{shi}.

\begin{cor} \thlabel{shicor1}  Let $C_0$ be the constant whose existence is guaranteed by  \thref{shiprop1}. Then there exists $c>0$ such that, for any $f: \mathbb{R}^{+}_{0} \rightarrow \mathbb{R}^{+}_{0}$ bounded away from $0$, and any $g: \mathbb{R}^{+}_{0} \rightarrow \mathbb{R}$ such that $g(t) \geq -f(t) \enskip \forall t \in \mathbb{R}^{+}_{0}$, we have, for all $t \geq 0$,
\[
\text{\textnormal{\prob{P}}} \Big(g(t) \leq X_t \leq g(t)+ C_0, \enskip \underline{X}_t \geq -f(t) \Big) \enskip \leq \enskip  
c \hspace{2pt} \frac{ \big\lbrace \big(f(t)+1\big) \wedge t^{1/2}\big\rbrace \big\lbrace \big(g(t)+f(t)+1\big) \wedge t^{1/2}\big\rbrace}{t^{3/2}},
\]
for all $t \geq 0$ where $x \wedge y := \min\{x,y\}$. In particular, there exists $c' > 0$ such that for all such $f$ and $g$ we have, for all $t \geq 0$,
\[
\qquad \text{\textnormal{\prob{P}}} \Big(X_t \leq g(t), \enskip \underline{X}_t \geq -f(t) \Big) \enskip \leq \enskip  
c' \hspace{2pt} \frac{\big\lbrace \big(f(t)+1 \big) \wedge t^{1/2}\big\rbrace \big\lbrace \big(g(t)+f(t)+1\big)^2 \wedge t \big\rbrace}{t^{3/2}}.
\]
\end{cor}

\vspace{5pt}

\begin{prop} \thlabel{nonzeroliminf}
Let $X$ be a L\'{e}vy process of the form $(Y_t - ct)_{t \geq 0}$, where $Y$ is a pure-jump subordinator and $c>0$. Assume that $X$ has zero mean and finite variance. For $\alpha>0$ let 
$X^{\alpha}_t := X_t + \alpha$. Then there exists $C>0$ such that, for any 
$f\colon [0,\infty) \rightarrow \mathbb{R}$ satisfying $\limsup_{t \rightarrow \infty} t^{-1/2}f(t)< \infty$ and  $f(t) \geq \alpha$, for all large $t$, we have 
\begin{equation} \label{eq:shia4levy}
\liminf_{t \rightarrow \infty} \hspace{3pt} t^{3/2} \hspace{3pt}\text{\textnormal{\prob{P}}} \Big(\underline{X}^{\alpha}_t \geq 0, \enskip \min_{t \leq s \leq 2t} X^{\alpha}_s \geq f(t), \enskip f(t) \leq X^{\alpha}_{2t} < f(t) + C\Big) \enskip > \enskip 0.
\end{equation} 
\end{prop}

\begin{pf} 
Let us assume that there exists no such constant $C>0$, and fix an $\alpha>0$. Select an $a>0$ corresponding to the choices $\epsilon = \frac{1}{2}$ and $\delta = 1$ in Lemma \ref{rcont}. Finally, choose an $h \in (0, \frac{1}{4} \min \{a, \frac{\alpha}{c}\})$. Define a random walk $(S_n)$ by $S_n := X_{nh}$ and note that $(S_n)$ satisfies the hypotheses of Lemma~A.3 in \cite{shi}. Let $K$ denote the positive constant corresponding to $(S_n)$ whose existence is guaranteed by Lemma A.3 in \cite{shi} (there, $K$ is called $2C$), and pick $\tilde{C} > K + 1 + \alpha$.  Since, in particular, we are assuming that \eqref{eq:shia4levy} does not hold for $C = \tilde{C}$, we infer the existence of a sequence $(t_k) \subseteq [0,\infty)$ such that $\lim_{k \rightarrow \infty} t_k = \infty$ with the property 
\begin{equation} 
\forall k \in \mathbb{N} \quad   \quad
 \left({\frac{t_k}{h}}\right)^{3/2} \prob{P} \Big(\underline{X}^{\alpha}_{t_k} \geq 0, \enskip \inf_{{t_k} \leq s \leq 2{t_k}} X^{\alpha}_s \geq f(t_k), \enskip f(t_k) \leq X^{\alpha}_{2t_k} < f(t_k) + \tilde{C} \Big) \enskip < \enskip \frac{1}{k} \enskip. \label{equ1}
 \end{equation}
Now define $n_k:= \lfloor{\frac{t_k - 1}{h}} \rfloor$. Note in particular that $(n_k+1) h \in [t_k - h, t_k]$; this will allow us to ensure that $X^{\alpha}_{t_k} \geq f(t_k)$ in the following computation. Define $a_{n_k} := f(t_k)+\alpha$ for each $k \in \mathbb{N}$, and $a_n := 0$ whenever there is no $k$ such that $n = n_k$. The important thing to note is that for any $j,k \in \mathbb{N}$ with $j \leq k$ and all $r \geq 0$ we have

$$ \min_{s \in [jh, jh +r]} X^{\alpha}_ s   \enskip \geq \enskip  S_j -rc + \alpha \enskip  \geq \enskip \underline{S}_k - rc + \alpha \enskip \geq \enskip \underline{S}_k \qquad \textrm{whenever} \qquad r \leq  \frac{\alpha}{c}. $$  

Consequently, whenever $ r \leq \frac{\alpha}{c}$ we find, for any $k \in \mathbb{N}$, that  
$\underline{X}^{\alpha}_{kh+r} \geq \underline{S}_k$
Recalling that $n_k h \in [t_k - 2h, t_k - h]$, we can write $t_k = n_kh + r $ for some $r_h \in [h,2h]$. Consequently,  we deduce that

$$ \underline{X}^{\alpha}_{t_k}\enskip  = \enskip \underline{X}^{\alpha}_{n_k h+r} \enskip \geq \enskip \underline{S}_{n_k} \qquad \textrm{provided}\qquad r_h \leq \frac{\alpha}{c},  $$

and the condition on $r_h$ holds because we have selected $h < \frac{\alpha}{2c}$. We will use this in the computation below, where we require $\{ \underline{S}_{n_k} \geq 0 \} \subseteq \{ \underline{X}^{\alpha}_{t_k} \geq 0 \}$. By the same considerations, we have also have the inclusion 
$$  \bigg\{ \inf\limits_{{t_k} \leq s \leq 2{t_k}} X^{\alpha}_s \geq f(t_k)  \bigg\} \enskip
\subseteq \enskip \bigg\{ \min_{n_k <j \leq 2n_k} S_j \geq  \enskip f(t_k)+ \alpha \bigg\} ,$$
since we have in fact picked $h < \frac{\alpha}{4c}$. We can therefore estimate 
\begin{align} \label{inequ2} \left({\sfrac{t_k}{h}}\right)^{3/2} \, &\prob{P} \Big(\underline{X}^{\alpha}_{t_k} \geq 0, \enskip \inf\limits_{{t_k} \leq s \leq 2{t_k}} X^{\alpha}_s \geq f(t_k), \enskip f(t_k) \leq X^{\alpha}_{2t_k} < f(t_k) + \tilde{C}\Big) \notag\\
&\geq 
\enskip n_k^{3/2} \, \prob{P} \, \Big(\underline{S}_{n_k} \geq 0, \enskip \min_{n_k <j \leq 2n_k} S_j \geq  \enskip f(t_k)+ \alpha, \enskip f(t_k) + \alpha \leq S_{2n_k} < f(t_k)+ \tilde{C} -1, \nonumber \\ 
& \hspace{150pt} \norm{X_t -X_{2n_k h}}{{t \in [2n_k h ,2t_k]}} < 1\Big) \nonumber \\ 
&\geq \enskip \sfrac12\, {n_k^{3/2}}\, \prob{P} \, \Big(\underline{S}_{n_k} \geq 0, \enskip \min_{n_k <j \leq 2n_k} S_j \geq  a_{n_k}, \enskip a_{n_k} \leq S_{2n_k} < a_{n_k} + K\Big).  
\end{align}

In the second inequality we used the fact that $h < \frac{a}{4}$ and the Markov property of $X^{\alpha}$ at time $2n_k h$. Combining \eqref{equ1} and \eqref{inequ2}, we find that, for all $k\in\mathbb{N}$, we have 
$$n_k^{3/2} \, \prob{P} \Big(\underline{S}_{n_k} \geq 0, \enskip \min_{n_k <j \leq 2n_k} S_j \geq  a_{n_k}, \enskip a_{n_k} \leq S_{2n_k} < a_{n_k} + K\Big) \enskip \leq \enskip \frac{2}{k}, $$
contradicting Lemma A.3 in \cite{shi}. \qed
\end{pf}

We finish this appendix with an arithmetic fact required in Section~4.

\begin{lem} \thlabel{sum}
For any $\alpha, \gamma>0$ and $k \in \mathbb{N}$ we have
$$\sum_{n = 4}^{\infty} n \left(   1 - \frac{1}{(\log n)^k}  \right)^{n^\alpha } \enskip < \enskip \infty.$$
\end{lem}

\begin{pf} It suffices to show that 
\[ 
\int_4^{\infty} x \left(   1 - (\log x)^{-k} \, \right)^{x^\alpha} dx 
\eneq    
\int_{\log 4}^{\infty} e^{2x} \left(   1 - x^{-k} \,  \right)^{e^{\alpha x}} dx      \enskip  < \enskip \infty . 
\]
To prove this integrability, we show that the second integrand is $o(e^{-x})$ as $x \rightarrow \infty$, or, equivalently, that $e^{\alpha x} \left( \log(x^k) - \log(x^k-1)  \right) - 3x   \enskip  \rightarrow \enskip  \infty$ 
as $x \rightarrow \infty$. For all $t > 1$  we have $\log' (s) \geq \frac{1}{t} \enskip \forall s \in [t-1,t]$, so $\log(x^k) - \log(x^k-1) \geq \frac{1}{x^k}$ for all $x > 1$. It remains to note that $x^{-k} e^{\alpha x} - 3 x \rightarrow \infty$ as $x \rightarrow \infty$.  \qed
\end{pf}


\bigskip

{\bf Acknowledgement:}  We thank Yan Fyodorov who suggested this project to us. We would like to thank three anonymous referees for their careful reading of  an earlier version of this paper. F.L. was supported by an EPSRC studentship for the duration of this project.

\vspace{20pt}

\newcommand{\bergaaa}[1]{}


\end{document}